\documentclass[11pt]{article}
\usepackage[a4paper, hmarginratio=1:1]{geometry} 
\usepackage{graphicx}
\usepackage{amssymb}
\usepackage{epstopdf}
\usepackage{mathtools}
\usepackage{tikz}
\usepackage{tikz-cd}
\usepackage{tkz-euclide}
\usepackage{subdepth} 
\usetkzobj{all} 
\usetikzlibrary{intersections}

\usepackage{comment}
\usepackage{subfig}
\usetikzlibrary{arrows,chains,matrix,positioning,scopes}
\usetikzlibrary{decorations.markings}

\usepackage[font=small, labelfont=bf, width=.8\textwidth]{caption}

\usepackage{amsmath}
\usepackage{amsthm}
\usepackage{cancel}
\usepackage{float}

\newtheorem{theorem}{Theorem}[section]
\newtheorem*{theoremint}{Theorem}

\newtheorem{proposition}[theorem]{Proposition}

\newtheorem{lemma}[theorem]{Lemma}
\newtheorem{corollary}[theorem]{Corollary}
\newtheorem*{corollaryint}{Corollary}
\theoremstyle{definition}
\newtheorem{definition}[theorem]{Definition}

\newtheorem{remark}[theorem]{Remark}

\usepackage[scaled]{helvet} 
\usepackage{courier} 
\usepackage{fourier} 
\normalfont
\usepackage[T1]{fontenc}
\let\phi\varphi

\usepackage{enumitem}
\setdescription{leftmargin=\parindent,labelindent=\parindent}

\newcommand{\gdots}{, \, \dots, \,}

\newcommand{\cat}[1]{${\rm CAT}(#1)$}

\usepackage[colorlinks=true, pdfstartview=FitV, linkcolor=blue, 
         citecolor=blue, urlcolor=blue]{hyperref}

\let\epsilon\varepsilon

\title{CAT(-1) metrics on small cancellation groups}
\author{Samuel Brown}

\begin{document}

\maketitle

\begin{abstract}
We give an explicit proof that groups satisfying a ``uniform'' $C'(1/6)$ small cancellation condition admit a geometric action on a \cat{-1} space. The proof consists of a direct construction of a piecewise hyperbolic structure on the presentation complex of such a group, together with folding moves to make the complex negatively curved. This argument is originally due to Gromov.
\end{abstract}

\section{Introduction}

A driving force behind modern geometric group theory is the search for negatively curved metrics on groups. Traditionally, the most popular notion of negative curvature has been word hyperbolicity, a coarse notion whose interplay with more local notions, such as the \cat{k} conditions, is not fully understood. Recently, much attention has been given to the \cat{0} condition, particularly in the context of \cat{0} cube complexes, and there has been a great deal of success in showing that certain families of hyperbolic groups act geometrically on \cat{0} cube complexes (see for example \cite{hsuwise10, wisemanu}). Often, the \cat{0} spaces obtained by this way are of very high dimension. An alternative approach is to look directly for \cat{-1} metrics on known spaces possessing geometric actions of groups. Although this approach is optimistic, it has the advantage of proving stronger bounds on the \cat{0} dimension of groups than the cubulation approach, as well as strengthening \cat{0} to \cat{-1}. 

In \cite{brown16}, we showed that hyperbolic limit groups were \cat{-1}, in particular of \cat{-1} dimension 2, by proving a combination theorem for a particular family of locally \cat{-1} spaces: 2-dimensional piecewise hyperbolic simplicial complexes satisfying the Link Condition. The same holds for hyperbolic graphs of free groups with cyclic edge groups, which were known to be \cat{0}, but not of dimension 2 \cite{hsuwise10}.

In the current paper, we apply similar techniques to groups satisfying a uniform version of the $C'(1/6)$ small cancellation condition. We are able to find an explicit piecewise hyperbolic metric on the presentation complex of such groups, and apply some folding moves to make this metric locally \cat{-1}. Specifically, a presentation is called \emph{uniformly $C'(1/6)$} if \emph{pieces} (overlaps between relators) are all shorter than a sixth of the length of the shortest relator. Our main theorem is then:

\begin{theoremint}[Theorem \ref{thm:smallcanccat-1}]
	Let $G$ be a group with a uniformly $C'(1/6)$ presentation. Then $G$ acts geometrically on a 2-dimensional \cat{-1} space.
\end{theoremint}


Although the uniform $C'(1/6)$ condition we use is in general stronger than the standard $C'(1/6)$ condition, it still holds for an important class of $C'(1/6)$ groups; namely, random groups in the density model at density $<1/12$. We therefore have the following corollary:

\begin{corollaryint}[Corollary \ref{cor:randomgroupscat-1}]
		Random groups in the density model, for density $d<1/12$, act geometrically on a 2-dimensional \cat{-1} space.
\end{corollaryint}

Wise showed in \cite{wise04} that $C'(1/6)$ groups are \cat{0}, and hence so are random groups at density $<1/12$. Ollivier and Wise then improved this to density $<1/6$ \cite{ollivierwise11}. However, since both results use cubulation, the \cat{0} spaces obtained are of high dimension, and so Corollary \ref{cor:randomgroupscat-1} represents an improvement in dimension as well as curvature.

\begin{remark}
	After completing this paper, we became aware that this argument was originally suggested by Gromov \cite{gromov01}, and a more general version of it (in the context of small cancellation over graphs of groups) is described in \cite{martin13}. The latter paper deals with a \cat{0} metric, but points out that the argument also works in the \cat{-1} case. We would like to thank Alexandre Martin and Anthony Genevois for bringing these two papers to our attention.
\end{remark}

I would like to acknowledge the support of the EPSRC for this work, and Henry Wilton for several helpful discussions.

\section{Preliminaries}

\subsection{Negatively curved complexes}

We will assume the reader is familiar with the definitions and basic theory of \cat{k} spaces; a good reference is \cite{bh}. In this paper, we focus on a specific family of \cat{-1} spaces:

\begin{definition}
	A \emph{piecewise hyperbolic simplicial complex} is a metric simplicial complex, each of whose $n$-simplices is isometric to an $n$-simplex in hyperbolic space $\mathbb{H}^n $.
\end{definition}

The following fundamental theorem is proved in \cite[Chapter II.5]{bh}:

\begin{theorem}[The Link Condition for piecewise hyperbolic simplicial complexes]
	Let $K$ be an piecewise hyperbolic simplicial complex with finitely many isometry types of simplices. Then $K$ is locally \cat{-1} if and only if for each vertex $v \in K$, $\text{link}(v,K)$ is \cat{1}.
\end{theorem}

\begin{remark}
	In the case that $K$ is 2-dimensional, links of vertices are topological graphs. These are \cat{1} precisely when they contain no closed geodesics of length $<2\pi$, making the link condition particularly easy to check in this case.
\end{remark}

\begin{definition}
	A piecewise hyperbolic simplicial complex is called a \emph{negatively curved complex} if it satisfies the Link Condition.
\end{definition}

To connect this to group theory, we make the following definition:

\begin{definition}
	A group is called \emph{\cat{-1}} if it acts properly discontinuously and cocompactly by isometries on a simply connected negatively curved complex. A group is called \emph{freely \cat{-1}} if it acts freely and geometrically on a simply connected negatively curved complex. The minimal dimension of such a complex is called the \emph{\cat{-1} dimension} of the group. 
\end{definition}

Note that a group is freely \cat{-1} if and only if it is the fundamental group of a negatively curved complex.

\subsection{Small cancellation conditions}

A good reference for classical small cancellation theory is \cite{ls77}, and we refer the reader there for full details. We only state here what is necessary for us to give our main theorem.

\begin{definition}
	Let $R = \{r_1 \gdots r_n\}$ be a set of cyclically reduced words on an alphabet $S \sqcup S^{-1}$, closed under taking cyclic permutations and inverses. A \emph{piece} in $R$ is a word $w$ which appears as an initial segment of at least two elements of $R$.
\end{definition}

\begin{definition} \label{def:uniformsmallcanc}
	Let $\mathcal{P}=\langle S \mid R \rangle$ be a presentation for a group $G$. Without loss of generality, assume $R$ is closed under taking cyclic permutations and inverses. We say $\mathcal{P}$ is $C'(1/6)$ of every piece in $R$ has length strictly less than $1/6$ of the length any relator in which it appears. Now let $g$ be the minimal length of any relator in $R$. We say $\mathcal{P}$ is \emph{uniformly $C'(\frac{1}{6})$} if every piece in $R$ has length strictly less than $g/6$, and moreover no element of $R$ is a proper power.
\end{definition}

\begin{remark}
	Groups which are $C'(\frac{1}{6})$ are torsion-free if and only if no relator is a proper power, and proper powers are forbidden by our uniform $C'(\frac{1}{6})$ condition. This torsion-freeness is necessary for our argument, since we produce a \emph{free} action on a \cat{-1} space, and all groups possessing such an action are torsion-free. However, we do not know whether the uniform small cancellation condition can be relaxed to the standard $C'(\frac{1}{6})$ condition in the torsion-free case.
\end{remark}

Our main theorem is the following:

\begin{theorem}\label{thm:smallcanccat-1}
	Let $G$ be a group with a uniformly $C'(1/6)$ presentation, $\langle S \mid R \rangle$. Then $G$ is \cat{-1} with \cat{-1} dimension 2.
\end{theorem}

\begin{remark}
	Random groups in the density model, for density $< 1/12$, satisfy the ordinary (non-uniform) $C'(1/6)$ condition \cite{gromov93}. Since they have all relations of equal length, they satisfy the uniform $C'(1/6)$ condition too. This provides an assurance that the uniform condition is not too much of a restriction; indeed we obtain the following immediate corollary of Theorem \ref{thm:smallcanccat-1}.
\end{remark}

\begin{corollary}\label{cor:randomgroupscat-1}
	Random groups in the density model, at density $<1/12$, are \cat{-1} with \cat{-1} dimension 2.
\end{corollary}

\section{Proof of the main theorem}

\subsection{Geometry of regular polygons}

The first part of our argument relies on choosing suitable metrics on the 2-cells in the presentation complex. These metrics are based on small regular hyperbolic polygons, however since sufficiently small hyperbolic polygons closely resemble Euclidean polygons, we will argue in the Euclidean case for technical simplicity. Proposition \ref{prop:hyperboliclikeeuclidean} makes explicit the conversion to a hyperbolic metric.

We first establish some terminology about such polygons.

\begin{definition}
	Let $P$ be a regular (hyperbolic or Euclidean polygon). A \emph{diagonal} is a geodesic connecting two (possibly consecutive) vertices of $P$. A \emph{segment} of $P$ is the smaller of the two pieces obtained by cutting $P$ along a diagonal (in the case where the diagonal is an edge of $P$, the segment is also this single edge). The diagonal is said to \emph{subtend} the corresponding segment. The \emph{length} of the segment is the number of edges of $P$ it contains, and the \emph{length} of the diagonal is the length of the corresponding segment. The \emph{radius} of a regular polygon is the distance from the centre to any boundary vertex. 
\end{definition}

\begin{definition}
	Let $P$ be an $n$-gon, and let $d_1$, $d_2$ be two diagonals of length less than $n/2$. Suppose $d_1$ and $d_2$ intersect at a point $p$. Precisely one of the connected components of $P - (d_1 \cup d_2)$ contains the centre of $P$, and the angle at $p$ which is inside this component is called the \emph{internal angle} between $d_1$ and $d_2$.
\end{definition}

\begin{lemma}\label{lemma:ngonangles}
	Consider two intersecting diagonals in a regular Euclidean $n+1$-gon, each of length at most $\lfloor n/6 \rfloor$. The minimal internal angle between such diagonals is $>2\pi/3$.
\end{lemma}

\begin{proof}
	Clearly, the case realising the minimal internal angle is where the two diagonals share an endpoint and are of the maximal permitted length. To see this, take any other intersection of diagonals $d_1$, $d_2$, where $d_i$ has endpoints $v_i$ and $w_i$. Without loss of generality, the clockwise order of the endpoints around the boundary of the polygon is $v_1$, $v_2$, $w_1$, $w_2$. If $d_1$ is not of maximal length, then increase its length by keeping $w_1$ fixed and moving $v_1$ anticlockwise. This clearly decreases the internal angle between $d_1$ and $d_2$. Similarly, increase the length of $d_2$ by fixing $v_2$ and moving $w_2$ clockwise; this also decreases the angle. Finally, fix $d_1$ and rotate $d_2$ by moving both $v_2$ and $w_2$ clockwise, until $v_2$ coincides with $w_1$. This process also decreases angle, and we have arrived at the extremal case.
	
	Since six consecutive maximal length diagonals fail to complete a hexagon, the internal angle between each pair must be $>2\pi/3$. See Figure \ref{fig:tikz19gon1}.
\end{proof}

\begin{figure}[H]
	\centering
\begin{tikzpicture}[scale=0.75]
    \draw[thick, black] \foreach \x in {0,1, ..., 21} {(\x*360/19:4cm) coordinate(\x) {}};
	\draw[thick, black] (0) \foreach \x in {1, ..., 19} { -- (\x){}};	

	\draw[black!70] (1) \foreach \x in {3,5,...,21} { -- (\x)};	
	\draw[black!70] (0) \foreach \x in {2,4,...,20} { -- (\x)};	
	\draw[black!70] (0) \foreach \x in {3,6,...,21} { -- (\x)};		
	\draw[black!70] (2) \foreach \x in {5,8,...,20} { -- (\x)};
		\tkzMarkAngle[arc=l,size=0.6cm,color=green,fill=blue, opacity = .5](7,4,1);
	\draw[thick, red] (1) -- (4) -- (7) -- (10) -- (13) -- (16) -- (19);
	\node[below, yshift = -12, xshift = -4] at (4) {$>2\pi/3$};
\end{tikzpicture}
	\caption{The extremal angle between diagonals of length $\leq3$ in a regular 19-gon. No internal angle between diagonals in the picture is less than the highlighted angle.}
	\label{fig:tikz19gon1}
\end{figure}
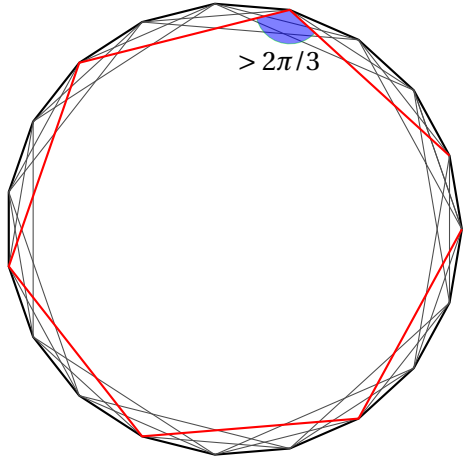

\begin{remark}\label{rmk:hyperboliclikeeuclidean}
	Lemma \ref{lemma:ngonangles} also applies to sufficiently small regular hyperbolic polygons, since the metric differs from a Euclidean metric by an arbitrarily small amount. The following proposition makes this more precise:
\end{remark}

\begin{proposition}\label{prop:hyperboliclikeeuclidean}
	For each integer $n>6$, there exists $r>0$ such that Lemma \ref{lemma:ngonangles} also holds for regular hyperbolic $n$-gons of radius $<r$.
\end{proposition}

\begin{proof}
	It is enough to show that, for any $n \geq 1$ and $s \geq 1$ we may pick $r$ small enough that the angle between two diagonals of length $n$, with a common endpoint, in a regular hyperbolic $(6n+s)$-gon of radius $r$, is $>2\pi/3$ (that is, the picture in Figure \ref{fig:tikz19gon1} still applies). Clearly, the extremal case is $s=1$. So, consider the right-angled triangle with one vertex at the centre $o$ of such a polygon, one vertex at the endpoint of a diagonal of length $n$, and one vertex at the midpoint of this diagonal. Denote by $\alpha$ and $\theta$ the two non-right angles in this triangle, as in Figure \ref{fig:tikzhypeucpropproof}. Our goal is to calculate the range of values of $r$ such that $\theta > \pi/3$. %
	\begin{figure}[h]
		\centering 
\begin{tikzpicture}
	\clip (-3,-1) rectangle (2,4);
    \draw[thick, black] \foreach \x in {0,1, ..., 21} {(\x*360/19:4cm) coordinate(\x) {}};

	\draw[thick, black] (0) \foreach \x in {1, ..., 19} { to[bend left=5] (\x){}};	

	\draw[very thin, red] (4) to[bend left=5] coordinate[pos=0.5](m) (7);
	\coordinate (o) at (0:0);
	\fill[black, opacity=0.1] (o) -- (4) to[bend left = 3] (m);
	\draw[thick, black] (0:0) to node[pos=0.5, right] {$r$} (4);
	\draw[thick, black] (0:0) -- (7);
	\draw[thick, black] (0:0) -- (m);
	\tkzMarkRightAngle (o,m,4);
	\tkzMarkAngle[arc=l, size=17pt, color=black, fill=black, opacity=.3](4,o,m);
	\tkzMarkAngle[arc=l, size=17pt, color=black, fill=black, opacity=.3](m,4,o);
	\draw[thick, red] (1) to[bend left=5]  (4) to[bend left=5] coordinate[pos=0.5](m) (7) to[bend left=5]  (10) to[bend left=5] (13) to[bend left=5] (16) to[bend left=5] (19);
	\node[below left, inner sep = 14pt] at (4) {$\theta$};
	\node[above, inner sep = 20pt] at (o) {$\alpha$};
	
\end{tikzpicture}
 \vspace{-1cm}
		\caption{A right-angled triangle in a regular hyperbolic $6n+1$-gon (for the case $n=3$.)}
		\label{fig:tikzhypeucpropproof} 
	\end{figure}
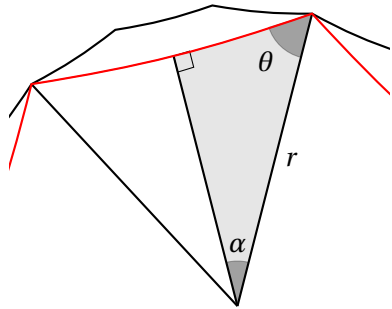	
	
	Since this triangle is obtained by bisecting the isosceles triangle with angle $2\pi \times \frac{n}{6n+1}$, we see that $\alpha = \frac{n\pi}{6n+1}$. It then follows from the second hyperbolic cosine rule that: \[
	\cosh(r) = \cot(\theta)\cot \left( \frac{n\pi}{6n+1} \right).
	\]
	
	We would like to find $r$ such that $\theta>\pi/3$; equivalently, $\cot(\theta) < 1/\sqrt{3}$. Hence: \begin{align*}
	\frac{\cosh(r)}{\cot(\frac{n\pi}{6n+1})}&<1/\sqrt{3} \\
	r&<\cosh^{-1} \left( \frac{1}{\sqrt{3}} \cot \left(\frac{n\pi}{6n+1} \right)    \right)
	\end{align*}
\end{proof}


Denote the right hand side of the above inequality by $r_{\text{max}}(n)$. This decreases fairly slowly with $n$. For illustration: $r_{\text{max}}(1) \approx 0.62$, $r_{\text{max}}(10) \approx 0.20$, $r_{\text{max}}(100) \approx 0.06$ and $r_{\text{max}}(1000) \approx 0.02$.

\subsection{Geometry of singular polygons}

A regular hyperbolic or Euclidean $n$-gon, with radius $r$, can be regarded as a simplicial complex with a vertex $o$ at the centre, and $n$ isometric isosceles 2-simplices with two sides of length $r$ and smallest angle $2\pi/n$ identified in a cycle around $o$. Denote this isosceles triangle by $T(n,r)$. 

	
For any integer $m$, we may obtain a singular 2-complex structure on a disc by identifying $m$ copies of $T(n,r)$ in the analogous manner. For $m<n$, the central vertex $o$ has local positive curvature: that is, its link is a loop of length $<2\pi$. For $m=n$, this is the usual regular $n$-gon, and for $m>n$, the central vertex $o$ has local negative curvature: the link has length $>2\pi$. We denote this singular disc by $D(m,n,r)$. See Figure \ref{fig:singularngon}.

\begin{figure}[hp]
	\centering
\begin{tikzpicture}
	\filldraw (0,0.5) circle (0pt) coordinate (o);
	\draw \foreach \x in {0,1,2,3,4,5,6,7} {
	(\x*45+0.15*45:4cm and 2cm) coordinate (\x) {} 
	(0,1) + (\x*45+0.65*45:4cm and 2cm) coordinate (\x1) {} };
	\draw (0) \foreach \x in {0,...,7} {-- (\x) -- (\x1)} -- cycle (0);
	\draw \foreach \x in {0,...,7} {(o)-- (\x) -- (\x1)};
	\draw \foreach \x in {0,...,7} {(o)-- (\x1)};
	\draw[fill=white, opacity=0.75] (o) -- (7) -- (71) -- cycle;
	\draw[fill=white, opacity=0.75] (o) -- (31) -- (4) -- cycle;
	\draw[fill=white, opacity=0.75] (o) -- (41) -- (5) -- cycle;
	\draw (4) -- (41);
\end{tikzpicture}
	\caption{A singular 16-gon. The 2-simplices are all isometric and isosceles, with angle (say) $\pi/6$ at the centre and radius $r$, so this is a picture of $D(16, 12, r)$.}
	\label{fig:singularngon}
\end{figure}
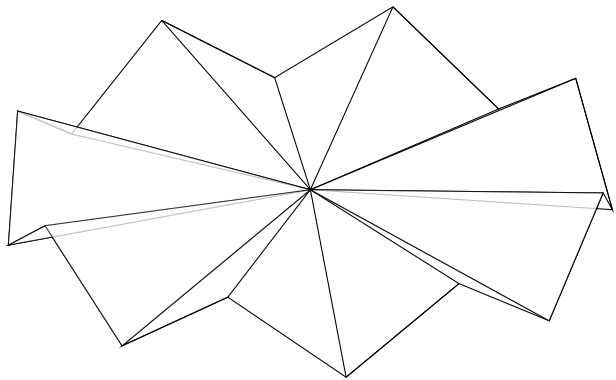

\begin{remark}
	For any $k<n/2$, we may define a segment of length $k$ in $D(m,n,r)$ in exactly the same way as for the regular $n$-gon $D(n,n,r)$. The isometry type of such a segment depends on $n$, $r$, and the underlying metric (i.e. hyperbolic or euclidean), but it does not depend on $m$. Moreover, Lemma \ref{lemma:ngonangles} still holds. See Figure \ref{fig:singularngonsegment}.
\end{remark}

\begin{figure}[hp]
	\centering	
\begin{tikzpicture}[scale=0.9]
	\filldraw (0,0.5) circle (0pt) coordinate (o);
	\draw \foreach \x in {0,1,2,3,4,5,6} {
	(-45+\x*43.5+0.1*43.5:4cm and 2cm) coordinate (\x) {}
	(0,1) + (-45+\x*43.5+0.6*43.5:4cm and 2cm) coordinate (\x1) {} };
	\draw (0) \foreach \x in {0,...,5} {-- (\x) -- (\x1)} -- (6);
	\draw \foreach \x in {0,...,5} {(o)-- (\x) -- (\x1)};
	\draw (o) -- (6);
	\draw (-45+6*43.5+0.1*43.5+360/11:4cm and 2cm) coordinate (A);
	\draw (-45+6*43.5+0.1*43.5+720/11:4cm and 2cm) coordinate (B);
	\%draw (-45+6*43.5+0.1*43.5+1080/11:4cm and 2cm) circle (1pt) coordinate (C);
	\draw (6) -- (A) -- (B) -- (0);
	\draw[fill = blue!20] (6) -- (A) -- (B) -- (0) -- cycle;
	\draw (o) -- (A);
	\draw (o) -- (B);
	\draw \foreach \x in {0,...,5} {(o)-- (\x1)};
	\draw[fill=white, opacity=0.75] (o) -- (41) -- (5) -- cycle;
	\draw[fill=white, opacity=0.75] (o) -- (51) -- (6) -- cycle;
	\draw[fill=white, opacity=0.75] (o) -- (1) -- (11) -- cycle;
	\draw[fill=white, opacity=0.75] (o) -- (0) -- (01) -- cycle;
	\draw (5) -- (51);
	\draw (1) -- (01);

	
\end{tikzpicture}
	\caption{A segment in a singular $D(m,n,r)$ is isometric to a segment of the same length in $D(n,n,r)$.}
	\label{fig:singularngonsegment}
\end{figure}
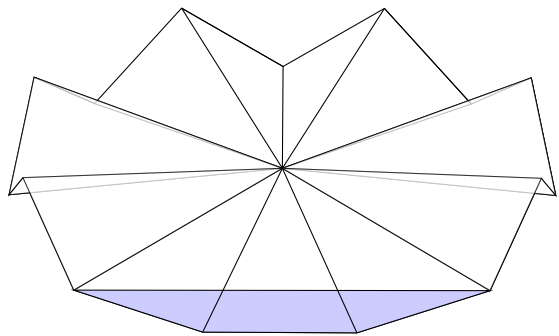
 
\subsection{Proof of the main theorem}

\begin{proof}[Proof of Theorem \ref{thm:smallcanccat-1}]
	By Definition \ref{def:uniformsmallcanc} no relator from $R$ is a proper power. Also assume, without loss of generality, that the relators are distinct up to cyclic permutation and inverses (if not, we can simply delete relators until this is the case). The small cancellation condition now reduces to the intuitive condition that, in the disjoint set of labelled cycles corresponding to the relators, the maximum length of any labelled path which appears at least twice is $<\frac{g}{6}$, where $g$ is the minimum length of any relator.
	
	Denote by $X$ the presentation complex corresponding to $\langle S \mid R \rangle$. This consists of a bouquet of circles $B$, labelled by $S$, and for each relator $r_i$ a disc $D_i$ whose boundary is attached to the path labelled by $r_i$. We equip $X$ with the following metric. Applying Proposition \ref{prop:hyperboliclikeeuclidean}, we may choose $r$ such that Lemma \ref{lemma:ngonangles} holds for the regular hyperbolic $g$-gon of radius $r$. For each $i \geq 1$, let $g_i=|r_i|$, and metrize each disc projecting to $D_i$ by a singular hyperbolic $g_i$-gon $D(g_i, g, r)$. Since the boundary of $D(g_i, g, r)$ is a $g_i$-cycle with all edges of length some constant $\lambda$, the metrics on each disc induce a well-defined metric on $X$. Moreover, by equipping each disc with the simplicial complex structure depicted in Figure \ref{fig:singularngon}, we get a simplicial complex structure on $X$. Denote by $Y=\widetilde{X}$ the universal cover of $X$ with the induced metric, and by $Z=\widetilde{B}$ the preimage in $\widetilde{X}$ of the bouquet $B$.
	
	A piece from the presentation corresponds to a maximal path in $Z$ which is contained in the boundary of two distinct discs (either two distinct lifts of the same $D_i$, or lifts of two different $D_i$). We refer to such paths also as ``pieces''. Each such piece subtends a segment of each these two discs, and these two segments are isometric by construction. Therefore, we may subdivide to make the segments into simplicial subcomplexes, and pass to a quotient space in which the two segments are identified. We will refer to such an identification as a ``fold''. 
	Note also that, after such an identification, the corresponding boundary relators still both bound well-defined, isometric discs; in particular, we can safely apply the same operation again to other segments whether or not they intersect the pair already folded. Of course, we must again subdivide $Y$ if we wish to retain a simplicial complex structure.
	
	So, let $\overline{Y}$ denote the quotient space obtained by identifying the corresponding pair of segments for every piece in $Y$. Denote by $\overline{Z}$ the image of $Z$ under this map. We claim that $\overline{Y}$ is a \cat{-1} space with a geometric action of $G$.
	

	Firstly, although we identify infinitely many pieces, $Y$ is a locally finite complex and there is an upper bound on the length of pieces. Thus there are only finitely many pieces containing each vertex of $Z$, and so $\overline{Y}$ is a well defined, locally finite, piecewise hyperbolic simplicial 2-complex. It is therefore sufficient to check the link condition on vertices. We will ignore vertices of degree 2 in the below analysis, since these are only artefacts of the subdivision process and do not affect the metric on links.
	
	There are two types of vertex to consider: those which are images of vertices in $Z$, and those which are not; i.e.\ images of points in the interior of discs of $Y$. So first, let $v'$ be a vertex in the image of $Z$, and let $v$ be the vertex of $Z$ mapping to $v$. This is unique, since vertices in $Z$ are never identified under $Y \rightarrow \overline{Y}$. 
	
	Topologically, $\text{link}_{Y}(v)$ is a graph with two vertices $s^{\pm}$ for each $s \in S \sqcup S^{-1}$, and an edge $s^{-}t^{+}$ whenever $st$ appears as a subword in any relator. These edges all have the same length, equal to the interior angles in the discs. There are additional vertices corresponding to the simplicial subdivision of the discs, but these are degree 2 and we ignore them.
	
	The quotient $Y \rightarrow \overline{Y}$ induces a map $\text{link}_{Y}(v) \rightarrow \text{link}_{\overline{Y}}(v')$, and we can describe this map very precisely. For each pair $st$ which appears as a subword of some piece (or whose inverse appears), there are at least two edges connecting $s^{-}$ and $t^{+}$; one edge for each relator in which $st$ or $t^{-1}s^{-1}$ appears. In $\text{link}_{\overline{Y}}(v')$, all of these edges are identified---folded---to a single edge. Moreover, any such occurence of multiple edges joining a pair of vertices corresponds to a subword of a piece, and so are all identified to a single edge in $\text{link}_{\overline{Y}}(v')$.
	
	The other operation performed by the map $\text{link}_{Y}(v) \rightarrow \text{link}_{\overline{Y}}(v')$ corresponds to the case where $v$ is the initial or final vertex of a piece in $Z$. Suppose there is a piece $p$ ending with $s$, and let $t_1 \gdots t_k$ be the set of all generators which occur as the subsequent letter to $p$ among the relators in which the piece appears. Then, for each $i$, there is an edge in $\text{link}_{Y}(v)$ from $s^{-}$ to $t_i^{+}$, and the map $\text{link}_{Y}(v) \rightarrow \text{link}_{\overline{Y}}(v')$ is a fold which has the effect of identifying a short initial subpath of each of these edges. This second operation is a homotopy equivalence; in particular, any unbased loop in $\text{link}_{\overline{Y}}(v')$ has a preimage (up to homotopy equivalence) in the graph obtained by identifying multiple edges in $\text{link}_{Y}(v)$. Lemma \ref{lemma:ngonangles} implies that, even if initial and final subpaths have been identified in this way, the length of the central path is still $>2\pi/3$. 
	
	We may therefore express the map from $\text{link}_{Y}(v) \rightarrow \text{link}_{\overline{Y}}(v')$ as a composition of two folding maps, one performing the identification of multiple edges, and one performing the identification of short initial segments of edges, as shown in Figure \ref{fig:link1}. The intermediate graph contains no bigons, and since the second map is a homotopy equivalence, all loops in the final graph must contain at least three unidentified central subpaths of edges, and thus have length $>2\pi$ as required.

	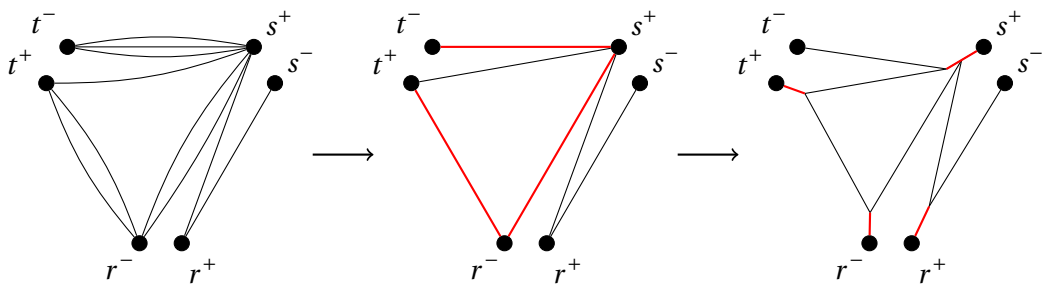
\begin{figure}[h]
		\centering	
\begin{tikzpicture}[every node/.style = {inner sep = 2pt, draw, fill=black, circle}, scale = 0.8]
	\begin{scope}[name = scope1]
	\node[label = 20:$s^-$] (s-) at (20:2cm) {};	
	\node[label = 40:$s^+$] (s+) at (40:2cm) {};	
	\node[label = 140:$t^-$] (t-) at (140:2cm) {};	
	\node[label = 160:$t^+$] (t+) at (160:2cm) {};
	\node[label = 260:$r^-$] (r-) at (260:2cm) {};	
	\node[label = 280:$r^+$] (r+) at (280:2cm) {};
	\draw[] (r-) to[bend left=10] (s+);
	\draw[] (r-) to[bend left=-5] (s+);
	\draw[] (r+) to[bend right=0] (s-);	
	\draw[] (r+) to[bend right=0] (s+);
	\draw[] (r-) to[bend left = 10] (t+);
	\draw[] (r-) to[bend right = 10] (t+);	
	\draw[] (t-) to[bend right = 10] (s+);	
	\draw[] (t-) to[bend right = 0] (s+);
	\draw[] (t-) to[bend right = -10] (s+);
	\draw[] (t+) to[bend right = 10] (s+);	
	\draw[thick, ->] (2.5,-0.5) to (3.5,-0.5);
	\end{scope}
	
	\begin{scope}[xshift = 6cm, name = scope2]
	\node[label = 20:$s^-$] (s-) at (20:2cm) {};	
	\node[label = 40:$s^+$] (s+) at (40:2cm) {};	
	\node[label = 140:$t^-$] (t-) at (140:2cm) {};	
	\node[label = 160:$t^+$] (t+) at (160:2cm) {};
	\node[label = 260:$r^-$] (r-) at (260:2cm) {};	
	\node[label = 280:$r^+$] (r+) at (280:2cm) {};
	\draw[thick, red] (r-) to[] (s+);
	\draw[] (r+) to[bend right=0] (s-);	
	\draw[] (r+) to[bend right=0] (s+);
	\draw[thick, red] (r-) to[] (t+);
	\draw[thick, red] (t-) to[bend right = 0] (s+);
	\draw[] (t+) to[bend right = 0] (s+);	
	\draw[thick, ->] (2.5,-0.5) to (3.5,-0.5);
	\end{scope}
	
	\begin{scope}[xshift = 12cm], name = scope3]
	\node[label = 20:$s^-$] (s-) at (20:2cm) {};
	\node[label = 40:$s^+$] (s+) at (40:2cm) {};	
	\node[label = 140:$t^-$] (t-) at (140:2cm) {};	
	\node[label = 160:$t^+$] (t+) at (160:2cm) {};
	\node[label = 260:$r^-$] (r-) at (260:2cm) {};	
	\node[label = 280:$r^+$] (r+) at (280:2cm) {};	
	\coordinate (tinner) at (160:1.5cm);	
	\coordinate (rinner) at (257:1.5cm);
	\coordinate (r+inner) at (295:1.5cm);
	\coordinate (s+inner) at (45:1.3cm);	
	\draw[] (s+inner) -- (s+) coordinate[midway](s+inner2);
	\draw[] (rinner) to (s+inner2);
	\draw[] (r+inner) to (s+inner2);
	\draw[] (r+inner) to (s-);	
	\draw[] (rinner) to (tinner);
	\draw[] (t-) to (s+inner);
	\draw[] (tinner) to (s+inner);
	\draw[thick, red] (s+inner) to (s+);
	\draw[thick, red] (r+inner) to (r+);
	\draw[thick, red] (tinner) to (t+);
	\draw[thick, red] (rinner) to (r-);
	\end{scope}
\end{tikzpicture}
		\caption{A planar example of the map on links of vertices of the first type. Thick red paths show regions where the map just applied was non-injective.}
		\label{fig:link1}
	\end{figure}


	We now address the second type of vertex in $\overline{Y}$. Let $v' \in \overline{Y}$ be a vertex which is the image under $Y \rightarrow \overline{Y}$ of a point in $Y - Z$. If $v'$ is the image of one of the singular points in the centre of a disc, then its link is a circle of length $\geq 2\pi$. This is because no segments intersect this point, so it is unaltered by the map $Y \rightarrow \overline{Y}$. Hence we may assume $v'$ is not of this type.

	
	Let $v_1 \gdots v_k$ be the set of all points in $Y$ mapping to $v_i$; note that this is indeed finite by local finiteness of $Y$. For each $i$, let $D_i$ be the disc in $Y$ containing $v_i$ in its interior. The link $L=\text{link}_{\overline{Y}}(v')$ is a quotient of a disjoint union of $k$ round circles $C_1 \sqcup \dots \sqcup C_k$.
	
	 The quotient map $\phi \colon C_1 \sqcup \dots \sqcup C_k \rightarrow L$ is induced by the process of identifying segments in $Y$. If such a segment contains one of the vertices $v_i$, then either $v_i$ is contained in the diagonal bounding the segment or in its interior. The map is therefore a multiple composition of two possible operations: identifications of subarcs of length $\pi$ between the $C_i$, or identifications of complete circles $C_i$. The latter operation does not affect $L$ and so we can assume that it does not occur. We now show that $L$ cannot contain any closed geodesics of length $<2\pi$. This is trivial in the case $k=1$, so assume $k \geq 2$.
	 
	 The map $\phi$ consists of repeatedly identifying length $\pi$ subarcs of different $C_i$. Now, each $v_i$ is the point of intersection of a number of diagonals bounding segments in $D_i$. The intersection of all these segments is a polygonal region in $D_i$ bounded by two of these diagonals, and it follows from Lemma \ref{lemma:ngonangles} that the subarc $\alpha_i$ of $C_i$ corresponding to this intersection has length at least $2\pi/3$ (see Figure \ref{fig:tikzck}). 
	 
	 Each identification of arcs under $\phi$ is induced by an isometry between segments in discs $D_i$ and $D_j$. It follows that $\phi$ isometrically identifies the arcs $\alpha_i$ for all $i$. Refer to the image of these arcs in $L$ as $\alpha$. In particular, the set $C_1 \cap \dots \cap C_n \subset L$ is nonempty. 
	 
	 Also, for each $i$, there is an open arc in $L$ which is contained in the image of only $C_i$, and not $C_j$ for $j \neq i$. Call this arc $\beta_i$ (see Figure \ref{fig:tikzck}), using the same name to refer to the image in $L$ or the subarc in $C_i$. 
	 
	 
	 The intersection of any subset of the $C_i$ is connected (since it contains $\alpha$). Therefore, the intersection of any subset of size $\geq 2$ of the $C_i$ is simply connected; it is a proper connected subspace of a circle. It follows that the subgraph $L_0 = \bigcup_{i \neq j} \left( C_i \cap C_j \right) = L - \bigcup_{i}\beta_i$ is simply connected.

	\begin{figure}[h]
		\centering	
\begin{tikzpicture}
    \draw (210:6cm) coordinate (-1) (230:6cm) coordinate (0) (250:6cm) coordinate (1) (270:6cm) coordinate (2) (290:6cm) coordinate (3) (310:6cm) coordinate (4) (330:6cm) coordinate (5) (350:6cm) coordinate (6);
	\draw[thick] (-1) -- (0) -- (1) -- (2) -- (3) -- (4) -- (5) -- (6);
	\filldraw[black, opacity = .1] (0) -- (1) -- (2) -- (3) -- (0);	
	\filldraw[black, opacity = .1] (1) -- (2) -- (3) -- (4) -- (5) -- (1);
	\draw[thick, name path = path1] (0) -- (3);
	\draw[thick, name path = path2] (1) -- (5);
	\draw [name intersections={of=path1 and path2, by=p}];
	\fill (p) circle (1.5pt) node[above] {$v_i$};
	\draw[thick] (p) circle (0.5cm);
	\draw[very thick, green!50!black] (p)+(20:0.5cm) arc (20:170:0.5cm);
	\node[above right, inner sep = 12pt, color = green!50!black] at (p) {$\beta_i$};
	\draw[very thick, green!30!black] (p)+(-10:0.5cm) arc (-10:-160:0.5cm);
	\node[below right, inner sep = 12pt, color = green!30!black] at (p) {$\alpha_i$};
\end{tikzpicture}
		\caption{The link of the vertex $v_i$ in the disk $D_i$. Lemma \ref{lemma:ngonangles} implies that both arcs $\alpha_i$ and $\beta_i$ have length $>2\pi/3$.}
		\label{fig:tikzck}
	\end{figure}
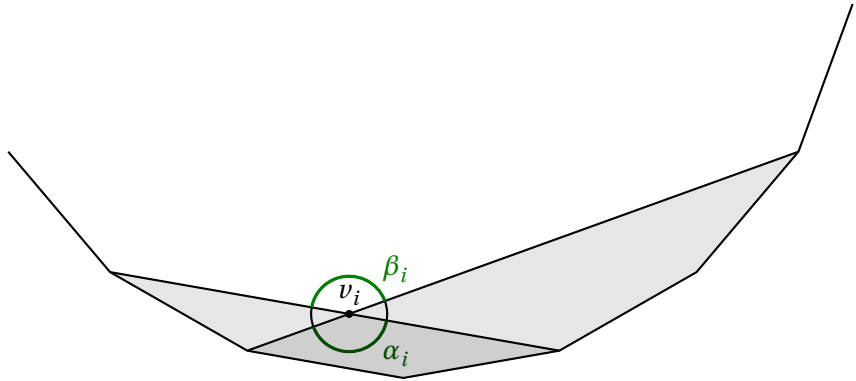

	 Now let $\ell$ be a geodesic loop in $L$. Since $L_0$ is simply connected, $\ell$ contains at least one $\beta_i$. If $\ell$ contains only one $\beta_i$ (say, $\beta_1$), then $C_1 - l_1$ is a geodesic path of length $2\pi - |\beta_i|$ between the endpoints of $\beta_1$ in $L_0$, and so must coincide with $\ell$ by simple connectedness of $L_0$; hence $\ell$ coincides with $C_1$ and has length $2\pi$. If $\ell$ contains at least three $\beta_i$, then it has length $>2\pi$ since each $\beta_i$ has length $>2\pi/3$. 
	 
	 The remaining case is that $\ell$ contains precisely two $\beta_i$; say $\beta_1$ and $\beta_2$.  Since $L_0$ is a tree, each component of $\ell - \{\beta_1 \cup \beta_2\}$ is contained in $C_1 \cup C_2$, and hence $\ell \subset C_1 \cup C_2$. 
	  Now $C_1 \cup C_2 \subset L$ is obtained from the circles $C_1$ and $C_2$ by identifying arcs of length $\pi$, all containing the common arc $\alpha$. If only a single arc of length $\pi$ is identified, then the 1-complex $C_1 \cup C_2$ clearly contains no loops of length $<2\pi$, in particular $\ell$ has length at least $2\pi$. Otherwise, $C_1$ and $C_2$ intersect in $L$ an arc longer than $\pi$. In this situation, there must be two intersecting segments in $D_1$ which are identified respectively with two overlapping segments in $D_2$. Therefore, there is a larger piece whose corresponding pair of segments was not identified (see Figure \ref{fig:ngonoverlap}). This contradicts the construction of the map $Y \rightarrow \overline{Y}$.

\begin{figure}[h]
	\centering	
\begin{tikzpicture}
	\begin{scope}
    \draw \foreach \x in {3, 4, ..., 10} { (\x*360/19:4cm) coordinate(\x) {} };	
	\coordinate(o) at (0:0);
	\filldraw[thick, red, opacity = 0.3] (6) -- (9) -- (8) -- (7) ;
	\filldraw[thick, blue, opacity = 0.3] (6) -- (5) -- (4) -- (7) ;
	\draw[thick, black] (4*360/19:4cm) \foreach \x in {3, 4, ..., 10} { -- (\x*360/19:4cm) coordinate(\x) {}} ;
    \draw[thin, black!20] \foreach \x in {3, 4, ..., 10} { (0:0) -- (\x*360/19:4cm)};
	\draw[thick, name=line1] (4) -- (7) node[below, pos = 0.79]  {$v_1$};
	\draw[thick, name=line2] (6) -- (9);
	\draw[dashed, fill = black, opacity = 0.2] (4) -- (5) -- (6) -- (7) -- (8) -- (9) -- (4);
	\end{scope}
	
	\begin{scope}[xshift = 7cm]
    \draw \foreach \x in {3, 4, ..., 10} { (\x*360/19:4cm) coordinate(\x) {} };	
	\coordinate(o) at (0:0);
	\filldraw[thick, red, opacity = 0.3] (6) -- (9) -- (8) -- (7) ;
	\filldraw[thick, blue, opacity = 0.3] (6) -- (5) -- (4) -- (7) ;
	\draw[thick, black] (4*360/19:4cm) \foreach \x in {3, 4, ..., 10} { -- (\x*360/19:4cm) coordinate(\x) {}} ;
    \draw[thin, black!20] \foreach \x in {3, 4, ..., 10} { (0:0) -- (\x*360/19:4cm)};
	\draw[thick, name=line1] (4) -- (7) node[below, pos = 0.79]  {$v_2$};
	\draw[thick, name=line2] (6) -- (9);
	\draw[dashed, fill = black, opacity = 0.2] (4) -- (5) -- (6) -- (7) -- (8) -- (9) -- (4);
	\end{scope}

\end{tikzpicture}
	\caption{If the red and blue pairs of segments are identified, then the union of the corresponding pieces is also a piece, and so the grey segments are also identified.}
	\label{fig:ngonoverlap}
\end{figure}
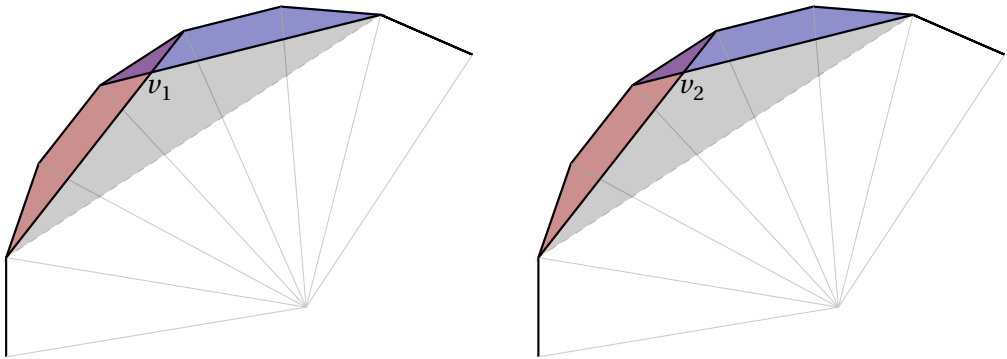

	  Hence, $\overline{Y}$ is \cat{-1}. Since the action of $G$ takes pieces to pieces, the invariance of the metric under the action of $G$ is clear, and the fact that the action is geometric follows from the fact that $Y$ was the presentation complex. This completes the proof.
	
\end{proof}

\begin{remark}\label{rmk:smallcanccat-1volume}
	Proposition \ref{prop:hyperboliclikeeuclidean} enables us to compute an approximate volume for the negatively curved complex constructed when we prove Theorem \ref{thm:smallcanccat-1}. The radius of each disc used in the construction is approximately $r_{\text{max}}(g/6)$, and the area of $D(g_i, g, r)$ is approximately $\pi r^2 g_i/g$ (approximating a flat polygon as a Euclidean disc). Hence, the area of the metrized presentation complex, before any folding is carried out, is approximately equal to \[A \approx \frac{ \Sigma_i g_i}{g} \times \pi r_{\text{max}}(g/6)^2,\]
	where \[
	r_{\text{max}}(n) = \cosh^{-1} \left( \frac{1}{\sqrt{3}} \cot \left(\frac{n\pi}{6n+1} \right)    \right).
	\]
	Of course, this is a slight overestimate, and decreases when folding is applied in a way that depends precisely on the pieces of the presentation.
\end{remark}

\bibliographystyle{alpha}
\bibliography{../phdbib}

\begin{thebibliography}{Mar13}

\bibitem[BH99]{bh}
Martin~R. Bridson and Andr{{\'e}} Haefliger.
\newblock {\em Metric spaces of non-positive curvature}, volume 319 of {\em
  Grundlehren der Mathematischen Wissenschaften [Fundamental Principles of
  Mathematical Sciences]}.
\newblock Springer-Verlag, Berlin, 1999.

\bibitem[Bro16]{brown16}
Samuel Brown.
\newblock A gluing theorem for negatively curved complexes.
\newblock {\em Journal of the London Mathematical Society}, 93(3):741--762,
  2016.

\bibitem[Gro93]{gromov93}
M.~Gromov.
\newblock Asymptotic invariants of infinite groups.
\newblock In {\em Geometric group theory, {V}ol.\ 2 ({S}ussex, 1991)}, volume
  182 of {\em London Math. Soc. Lecture Note Ser.}, pages 1--295. Cambridge
  Univ. Press, Cambridge, 1993.

\bibitem[Gro01]{gromov01}
M.~Gromov.
\newblock {${\rm CAT}(\kappa)$}-spaces: construction and concentration.
\newblock {\em Zap. Nauchn. Sem. S.-Peterburg. Otdel. Mat. Inst. Steklov.
  (POMI)}, 280(Geom. i Topol. 7):100--140, 299--300, 2001.

\bibitem[HW10]{hsuwise10}
Tim Hsu and Daniel~T. Wise.
\newblock Cubulating graphs of free groups with cyclic edge groups.
\newblock {\em Amer. J. Math.}, 132(5):1153--1188, 2010.

\bibitem[LS77]{ls77}
Roger~C. Lyndon and Paul~E. Schupp.
\newblock {\em Combinatorial group theory}.
\newblock Springer-Verlag, Berlin-New York, 1977.
\newblock Ergebnisse der Mathematik und ihrer Grenzgebiete, Band 89.

\bibitem[Mar13]{martin13}
Alexandre Martin.
\newblock Complexes of groups and geometric small cancellation over graphs of
  groups.
\newblock Preprint, 2013.

\bibitem[OW11]{ollivierwise11}
Yann Ollivier and Daniel~T. Wise.
\newblock Cubulating random groups at density less than {$1/6$}.
\newblock {\em Trans. Amer. Math. Soc.}, 363(9):4701--4733, 2011.

\bibitem[Wis04]{wise04}
D.~T. Wise.
\newblock Cubulating small cancellation groups.
\newblock {\em Geom. Funct. Anal.}, 14(1):150--214, 2004.

\bibitem[Wis12]{wisemanu}
Daniel~T. Wise.
\newblock The structure of groups with a quasiconvex hierarchy.
\newblock Preprint. http://www.math.mcgill.ca/wise/papers.html, 2012.

\end{thebibliography}

\end{document}